\documentclass{article}

\usepackage{amsmath, amsthm, amsfonts, amssymb}
\usepackage{mathptmx}
\usepackage{graphicx}
\usepackage{fancybox}
\usepackage{color}
\usepackage{ulem}

\def\supp{\operatorname{supp}}

\def\a{\alpha}
\def\e{\varepsilon}

\newtheorem{exam}{Example}[section]
\newtheorem{thm}{Theorem}[section]
\newtheorem{df}{Definition}[section]

\newtheorem{lem}{Lemma}[section]

\newtheorem{prop}{Proposition}[section]
\newtheorem{rem}{Remark}[section]
\numberwithin{equation}{section}
\newtheorem{rema}{Remark A.$\!\!$}
\newtheorem{lema}{Lemma A.$\!\!$}
\newtheorem{propa}{Proposition A.$\!\!$}
\newtheorem{dfa}{Definition A.$\!\!$}
\newcommand{\dstl}{\displaystyle}
\newcommand{\tstl}{\textstyle}

\newcommand{\wt}[1]{\widetilde{#1}}

\newcommand{\wbl}[1]{\:\!\overline{\;\!\!{#1}}}
\newcommand{\wbm}[1]{\;\!\overline{\:\!\!{#1}}}
\newcommand{\wbs}[1]{\,\overline{\!{#1}}}

\newcommand{\wbml}[1]{\;\!\overline{\:\!\!{#1}\;\!\!}}

\newcommand{\wbsm}[1]{\,\overline{\!{#1}\:\!\!}}

\newcommand{\ubm}[1]{\underline{{#1}\;\!\!}\;\!}
\newcommand{\ubs}[1]{\underline{{#1}\;\!\!\!}\;}
\newcommand{\ubb}[1]{\underline{\:\!{#1}\;\!\!\!}\,}

\newcommand{\ch}[1]{\chi\kern-.05em\lower1ex\hbox{$\scriptstyle{#1}$}}
\newcommand{\chb}[1]
{\wbml{\chi}\kern-.02em\lower1ex\hbox{$\scriptstyle{#1}$}\;\!}

\newcommand{\vLambda}{{\mit\Lambda}}

\newcommand{\vPhi}{{\mit\Phi}}
\newcommand{\vPsi}{{\mit\Psi}}
\newcommand{\vOmega}{{\mit\Omega}}
\newcommand{\emp}{\;\!{\rm o}\!\!\!/\;\!}
\newcommand{\nequiv}{\,\,{\equiv}\!\!\!\! /\,\,\,}

\newcommand{\vphi}{\varphi}
\newcommand{\bs}{\:\!\!\setminus\:\!\!}

\newcommand{\Lap}{{\mit{\Delta}}}

\newcommand{\R}{{\bf{R}}}
\newcommand{\N}{{\bf{N}}}

\newcommand{\BC}{{B\;\!\!C}}
\newcommand{\Rn}{{\bf R}\kern-0.08em\lower-0.75ex\hbox{$\scriptscriptstyle n$}}
\newcommand{\Rd}[1]{{\bf R}\kern-0.1em\lower-0.75ex\hbox{$\scriptscriptstyle{#1}$}}
\newcommand{\Sd}{S\kern0.0em\lower-0.75ex\hbox{$\scriptscriptstyle d\;\!\!-\:\!\!1$}}
\newcommand{\Sn}{S\kern0.0em\lower-0.75ex\hbox{$\scriptscriptstyle n\;\!\!-\:\!\!1$}}
\newcommand{\Cn}{{\bf C}\kern-0.01em\lower-0.75ex\hbox{$\scriptscriptstyle n$}}
\newcommand{\el}[1]{{\ell}\kern0.02em\lower-0.75ex\hbox{$\scriptscriptstyle{#1}$}}
\newcommand{\EL}[1]{{L}\kern0.0em\lower-0.75ex\hbox{$\scriptscriptstyle{#1}$}}
\newcommand{\EA}[1]{{A}\kern0.0em\lower-0.75ex\hbox{$\scriptscriptstyle{#1}$}}
\newcommand{\Tst}{T\kern0.0em\lower-0.58ex\hbox{$\scriptscriptstyle *$}}

\newcommand{\wtM}[1]{{\:\wt{\;\!\!\!M\;\!\!}}
\kern0.1em\lower-0.9ex\hbox{$\scriptstyle{#1}$}}
\newcommand{\wbM}[1]{{\;\!\wbs{\:\!\!M\:\!\!}}
\kern0.14em\lower-0.9ex\hbox{$\scriptstyle{#1}$}}

\newcommand{\wbY}{{\;\!\wbm{Y\:\!\!}}}

\newcommand{\tle}{\;\!\tilde{\:\!\!\e\;\!\!}\:\!}

\title{ On general  Caffarelli-Kohn-Nirenberg type inequalities involving non-doubling
weights }
\author{   Toshio Horiuchi}

\begin{document}
\maketitle

\begin{abstract}
By using $W(\R_+)= P(\R_+) \cup Q(\R_+)$ as a  class of weight functions, we will establish 
the Caffarelli-Kohn-Nirenberg type inequalities  with  non-doubling weights being  permitted.  
The classical  Caffarelli-Kohn-Nirenberg type inequalities are categorized  into   non-critical  and  critical cases,
  and it  is known  that   there  is some  kind of    mysterious relationship between them.  
Interestingly the  new framework in this treatise allows them to be integrated and reveals the  meaning of  mysterious relationships.

\end{abstract}

\section{Introduction }The  main purpose of the present paper is 
to study the Caffarelli-Kohn-Nirenberg type inequalities, which are abbreviated as  the CKN-type inequalities.
We will  establish the   CKN-type inequalities  with  non-doubling weights.
For this purpose we introduce a class of weight functions denoted by $W(\mathbf R_+)$  $(\mathbf R_+=(0,\infty))$, that is 
$$W(\mathbf R_+)= \{ w\in C^1({\R}_+): w>0, \lim_{t\to+0}w(t)=a\, \text{  for  some }\,a\in [0,\infty] \}.$$  
Further  we define two subclasses of $W(\R_+)$, that  is 
 \begin{equation}  \begin{cases}&P(\mathbf R_+)= \{ w(t)\in W(\mathbf R_+) : \,  w(t)^{-1} \notin L^1((0,\eta)) \, \text{ for some} \, \eta >0\},\\
 &Q(\mathbf R_+) =\{ w(t)\in W(\mathbf R_+) :  \, w(t)^{-1}\in L^1((0,\eta)) \, \text{ for any } \, \eta >0 \}.
 \end{cases}
 \end{equation}
Clearly   $W(\mathbf R_+)=P(\mathbf R_+) \cup Q(\mathbf R_+)$ and $ P(\mathbf R_+) \cap Q(\mathbf R_+)=\phi$ hold. 
 See Section 2 for  the precise definition of  theses classes.
A positive continuous function $w(t)$ on $(0,\infty)$ is  said to be a  doubling weight  if  there exists a positive number $C$  such  that  we  have 
$C^{-1} w(t)\le w(2t)\le C w(t) \, (0<t<\infty)$,
where $C$ is independent of  each $t \in (0,\infty)$.
If $w(t)$ does not possess this property, then $w(t)$ is  said to be  a non-doubling weight,
 and typically 
$e^{-1/t} \in  P(\mathbf R_+)$  and  $e^{1/t} \in Q(\mathbf R_+)$ are non-doubling weights. It  will be   seen  that
our   results on  the CKN-type  inequalities    essentially depend on whether $w$ belongs to
 $ P(\mathbf R_+) $ or $ Q(\mathbf R_+)$.
\par\medskip
The classical  CKN-type inequalities are categorized  into 
 the non-critical inequalities  (\ref{1.4}) and the critical inequalities  (\ref{CKNc}), and  there is some kind of   mysterious relationship between them. For  the details see  (\ref{relations}) in Remark \ref{remark2.2} and Proposition \ref{relation} in Section 7.
By using a new   frame-work,   we will show that they can be treated in a unified manner,  and 
as a result  we  will  make clear the meaning of the   relationship. 
Let us  explain the situation in a little more detail.
We shall establish the  following CKN-type inequalities (\ref{mistake}) with non-doubling weights  that contain the classical   CKN-type inequalities. 
    \par\medskip
Let  $w\in W(\mathbf R_+)$, 
 $1<p\le q <\infty$, $\eta>0$, $\mu>0$ and  $ 0 \le 1/p -1/q \le1/n$.
Then, there exists a positive number $C_n =C_n(p,q,\eta,\mu, w)$ such  that  we have  
\begin{equation}
\int_{B_\eta} |\nabla u|^p w(|x|)^{p-1}|x|^{1-n}\,dx\ge  C_n \left(  \int_{B_\eta} 
\frac{ |u|^q |x|^{1-n}\,dx}{ w(|x|)f_\eta (|x|)^{1+q/{p'}}} \right)^{p/q}, \quad u\in C_c^\infty(B_\eta\setminus \{0\}),\label{mistake}
\end{equation} 
where $p'= p/(p-1)$, $B_\eta$ is  the  ball $ \{ x\in \R^n : \: |x|< \eta\}$   and 
\begin{equation}f_\eta(t)= \begin{cases} & \mu+ \int_t^\eta \frac{1}{w(s)}\,ds,\qquad 
w\in P(\R_+),\\
&\int_0^t \frac{1}{w(s)}\,ds, \qquad \qquad
w\in Q(\R_+).
\end{cases}
\end{equation}
The inequalities  (\ref{mistake}) clearly involve  non-doubling weights. We  remark  that  when  $p=q$,   they are reduced to the  Hardy-type inequality and have been  established in \cite{ho3}; Theorem 3.1. (cf. \cite{ho4})\par
For now, let's admit the inequalities (\ref{mistake}) and show how they unify non-critical and critical CKN-type inequalities.
For this  purpose,
we  set   $w(t)= t^{p'\gamma+ 1}, \gamma\in \R $.
By  the  definitions of $P(\R_+)$, $Q(\R_+)$  and $w(t)$,  we  immediately 
 see that $w(t)\in P(\R_+)$ if  $\gamma\ge  0$, and $w(t)\in Q(\R_+)$ if  $\gamma< 0$. Since $w(|x|)^{p-1} |x|^{1-n}= |x|^{p(1+\gamma)-n}$ holds,  we are  able  to  check   that

\begin{equation}
f_\eta (t) =\begin{cases} &\frac1{p'\gamma}t^{-p'\gamma}, \quad\qquad \mbox{ if } \gamma> 0  \mbox{ and } \mu=  \frac1{p'\gamma}\eta^{-p'\gamma},\\
\\
& \mu+  \log \eta /t, \quad\quad \mbox{ if } \gamma= 0,\\
\\
 & -\frac1{p'\gamma}t^{-p'\gamma},\qquad\; \mbox{ if } \gamma< 0,
\end{cases}
\end{equation}
and we  have
\begin{equation}
\frac{|x|^{1-n}}{ w(|x|)  f_\eta (|x|)^{1+q/{p'}}}= \begin{cases}
&(p'\gamma)^{1+q/p'} |x|^{q\gamma-n}, \qquad \mbox{ if } \gamma> 0  \mbox{ and } \mu=  \frac1{p'\gamma}\eta^{-p'\gamma},\\
\\
&\frac{1}{|x|^n (\log (R\eta/|x|))^{1+q/p'}}, \qquad\qquad\mbox{ if } \gamma= 0  \mbox{ and } R=  e^{\mu}, \\
\\
 & (-p'\gamma)^{1+q/p'} |x|^{q\gamma-n}, \quad\; \mbox{ if } \gamma< 0.   

\end{cases}\label{1.13}
\end{equation}

\par\noindent 
If   $\gamma\neq 0$, then  the inequalities 
 (\ref{mistake}) coincide with  the non-critical CKN-type  inequalities  (\ref{1.4}) with ${\R}^n $ replaced by $B_\eta$, and if $\gamma=0$, then 
they coincide with  the critical CKN-type  inequality (\ref{CKNc}).
Therefore,
 it follows from  (\ref{1.4}) and  (\ref{CKNc}), together with the remarks just after (\ref{symmetric}) and (\ref{symmetric2}), that  the  desired inequalities (\ref{mistake})
hold, provided that    $C_n$ satisfies $C_n(p'|\gamma|)^{p(1/q+1/p')} \le S^{p,q;\gamma}$ if $\gamma\neq 0$;  $C_n\le C^{p,q;R}$ if $\gamma=0$. 
Thus  we  see that the  inequalities (\ref{mistake}) unify the non-critical  and the critical CKN-type  inequalities.
\par\medskip
{ Here   we  clearly  state  that} the  inequalities (\ref{mistake})  do  not  hold for  $w\in W(\R_+)$ unconditionally, unless $n=1$ or $p=q$ (the Hardy-type). 
In order to study the validity of (\ref{mistake}) with each $w(t)\in W(\R_+)$
we  introduce  {\bf  the  non-degenerate condition (\ref{H-condition})} in  Section 3 concerning  the behavior of $w(t)$ near $t=0$, 
and   in Theorem \ref{th4.2}  we   establish the validity of  (\ref{mistake})  under  (\ref{H-condition}).
Roughly speaking,  (\ref{H-condition}) assures that  $w(t)$ does  not behave so  badly as $t\to +0$, and hence the function 
 $H(\rho)$, given  by Definition (\ref{Definition3.2}),  is bounded away from $0$. On  the contrary   if $\dstl\lim_{\rho\to +0}H(\rho)=0$ is  assumed, then   we   show  in Theorem \ref{Thm3.3} that  the inequalities (\ref{mistake}) fail to  hold,  and in  Theorem \ref{Thm3.4} we   characterize a  set of  weight functions 
for  which (\ref{H-condition}) is violated. 
\par\medskip

In \cite{CKN} they established general multiplicative inequalities with weights being  powers of distance from the origin. 
The inequalities they presented include the classical Hardy-Sobolev inequalities (cf.\cite{ho1, m,t1,t2}), which later became known as the CKN-type inequalities, are still the subject of many interesting studies. 
In the study  of the CKN-type inequalities, the presence of weight functions in the both sides 
prevents us from employing effectively the so-called spherically symmetric rearrangement.
Further the invariance of $\R^n $ by the group of dilatations creates some possible loss of compactness. 
 Partly because of these difficulties, it was a very interesting subject of research.
In \cite{ho2,CH,CH2,hk3} we have  also studied
these inequalities intensively to show that the existence of extremals, the values of best constants  and
their asymptotic behaviors  essentially depend upon the relations among  parameters in the inequalities.\par
Recently,
there are  authors who have studied classical inequalities such  as 
Hardy-Sobolev type  and CKN-type  employing  various transformations to obtain fruitful results (cf. \cite{Ioku,s,s2,s3, Z}). 
We have  also revisited  weighted Hardy's inequalities as follows:
In \cite{ah3,  ahl}  we improved them under one-sided boundary conditions, and 
in \cite{ah4,ho3,ho4} we introduced  a new frame-work  with a class of  weight functions $W(\R_+)$. This work is based on  them.
\par\medskip
This paper is organized in the following way: 
The proof  
Theorem \ref{th4.2} is given in Section {4}. 
Theorem \ref{Thm3.3} and Theorem \ref{Thm3.4} are established  in  Section 5 and in Section 6 respectively. 
  In Appendix we collect useful  relations among the best constants of the CKN-type  inequalities for  the  sake of 
self-containedness. 

\par\medskip 

 \section{Preliminaries  }
 \subsection{A review of the classical CKN-type inequalities}
 We  begin with  reviewing  the CKN-type inequalities as a background for this paper.
We  describe fundamental facts on the CKN-type inequalities  according to \cite{hk3}, and
 classify 
 the CKN-type inequalities into  two cases, according to  the  range of the parameter  $\gamma\in \R$. Namely
 \begin{df}\label{radialspace}
  The parameter $\gamma$ is said to be
critical  and non-critical  if  $\gamma$ satisfies $\gamma=0$  and  $\gamma \neq 0$  respectively.
 \end{df}
\begin{rem}
The non-critical case ($\gamma\neq 0$) was further  classified  in \cite{hk3}.
\end{rem}
\par
In the  non-critical case, the CKN-type  inequalities have the following form:
\begin{equation}\int_{\R^n}|\nabla u(x)|^p |x|^{p(1+\gamma )-n}\, dx\ge
 S^{p,q;\gamma}\bigg(\int_{\R^n}|u(x)|^q |x|^{\gamma q-n}\, dx\bigg)^{p/q}, \; u \in 
C_{\rm c}^{\infty}(\R^n\setminus \{0\}), 
\label{1.4}
\end{equation}
where  $\gamma \neq 0$,  $n \ge 1$, $1<p\le q <\infty$ and  $ 0 \le 1/p -1/q \le1/n$.
Here $S^{p,q;\gamma}= S^{p,q;\gamma}(\R^n) $ is  called the  best constant and given by the following variational problem:
\begin{align}
 S^{\:\!p\;\!\!,q\:\!;\:\!\gamma}
 &=\inf\{{\;\!}E^{\:\!p\;\!\!,q\:\!;\:\!\gamma}[u] : u\in
 C^{\:\!\infty}_{\rm c}(\R^{\;\!\!n}\bs{\;\!\!}\{0\}{\:\!\!}){\;\!\!}\bs{\;\!\!}\{0\}\}, \label{best}
\end{align}
where
\begin{equation}
 E^{\:\!p\;\!\!,q\:\!;\:\!\gamma}[u]
 =\dfrac{\int_{\R^n}|\nabla u(x)|^p |x|^{p(1+\gamma )-n}\, dx}
 {\bigg(\int_{\R^n}|u(x)|^q |x|^{\gamma q-n}\, dx\bigg)^{p/q}}{\qquad}\mbox{ for }u\in
 C^{\:\!\infty}_{\rm c}(\R^{\;\!\!n}\bs{\;\!\!}\{0\}{\:\!\!})\setminus \{0\}.
\end{equation}
%
 We  also define the  radial best constant  $S^{\:\!p\;\!\!,q\:\!;\:\!\gamma}_{\rm rad}=S^{\:\!p\;\!\!,q\:\!;\:\!\gamma}_{\rm rad}(\R^n)$ as follows.
 \begin{df}\label{df2.2}
Let 
$\vOmega$ be  a radially symmetric domain. 
For any function space  $V(\vOmega)$ on 
$\vOmega$,   we set
\begin{equation}\dstl{
 V(\vOmega)_{\rm rad}^{}=\{{\;\!}u\in V(\vOmega) : \mbox{$u$ is radial}{\;\!}\}.
}\end{equation}
\end{df}
 Then we define
\begin{equation}
S^{\:\!p\;\!\!,q\:\!;\:\!\gamma}_{\rm rad}
 =\inf\{{\;\!}E^{\:\!p\;\!\!,q\:\!;\:\!\gamma}[u] :  u
 \in C^{\:\!\infty}_{\rm c}(\R^{\;\!\!n}\bs{\;\!\!}\{0\}{\:\!\!})_{\rm rad}^{}{\!}
 \bs{\;\!\!}\{0\}\}. \label{symmetric}
\end{equation}
It follows  from Proposition \ref{classicalCKN} that
$\dstl{
 S^{\:\!p\;\!\!,q\:\!;\:\!\gamma}=S^{\:\!p\;\!\!,q\:\!;\:\!-\:\!\gamma} \text {and } S_{\rm{rad}}^{\:\!p\;\!\!,q\:\!;\:\!\gamma}=S_{\rm{rad}}^{\:\!p\;\!\!,q\:\!;\:\!-\:\!\gamma}.
}$
By $S^{\:\!p\;\!\!,q\:\!;\:\!\gamma}(\Omega)$  we  denote  the  best constant  with $\R^n$ replaced by $\vOmega$.
Here we remark  that
the  best  constants $S^{\:\!p\;\!\!,q\:\!;\:\!\gamma}\;(=S^{p,q;\gamma}(\R^n)) $ is invariant if  $\R^n$ is  replaced by an  arbitrary  domain $\vOmega$ containing the origin.  $ S^{\:\!p\;\!\!,q\:\!;\:\!\gamma}_{\rm rad}\;(=S^{\:\!p\;\!\!,q\:\!;\:\!\gamma}_{\rm rad}(\R^n)) $ is  also invariant  if $\R^n$ is  replaced by a  radially symmetric domain $\vOmega$ containing the origin.
Further $S^{p,q;\gamma} $ and $ S^{\:\!p\;\!\!,q\:\!;\:\!\gamma}_{\rm rad} $ are not attained by a function having compact support. 
(cf. \cite{ho2,po}) 
%

\par\medskip\noindent In the critical case that $\gamma=0$, the CKN-type inequalities have the following form: For $\eta>0$,
let $B_\eta$ be  the  ball $ \{ x\in \R^n : \: |x|< \eta\}$.   
\begin{equation}\label{CKNc}\dstl{
 {\int_{\;\!\!B_{\;\!\!\eta}^{}}\;\!\!}|\nabla u(x)|^{p}|x|^{p-n}{\:\!}dx
 \ge C^{p,q;R}\Big({\int_{\:\!\!B_{\;\!\!\eta}^{}}\;\!\!}
\dfrac{ |u(x)|^{q}}{|x|^n \left(\log(R\eta/|x|)\right)^{1+\:\!q/\;\!\!p'}\;\!\!}{\:\!}dx
 \Big)^{\:\!\!p/\;\!\!q}\,
},\quad  u \in 
C_{\rm c}^{\infty}(B_\eta\setminus \{0\}),\end{equation}
where $R$ is a positive  number satisfying $ R> 1$  and the best constant $C^{p,q;R}=C^{p,q;R}(B_\eta)$ is given by the  variational problrm:
\begin{align}
 C^{\:\!p\;\!\!,q\:\!;\:\!R}
 &=\inf\{{\;\!}F^{\:\!p\;\!\!,q\:\!;\:\!R}[u] : u\in
 C^{\:\!\infty}_{\rm c}(B_\eta\bs{\;\!\!}\{0\}{\:\!\!}){\;\!\!}\bs{\;\!\!}\{0\}\}, 
\end{align}
where
\begin{equation}
 F^{\:\!p\;\!\!,q\:\!;\:\!R}[u]
 =\dfrac{\int_{B_\eta}|\nabla u(x)|^p |x|^{p-n}\, dx}
 {\Big({\int_{\:\!\!B_{\;\!\!\eta}^{}}\;\!\!}
{ |u(x)|^{q}}{|x|^{-n }\left(\log(R\eta/|x|)\right)^{-1-\:\!q/\;\!\!p'}\;\!\!}{\:\!}dx
 \Big)^{\:\!\!p/\;\!\!q}}{\,\,\,}\mbox{ for }u\in
 C^{\:\!\infty}_{\rm c}(B_\eta\bs{\;\!\!}\{0\}{\:\!\!})\setminus \{0\}.
\end{equation}
We also define the  radial best constant
\begin{align}
  C^{\:\!p\;\!\!,q\:\!;R}_{\rm rad}
 &
 =\inf\{{\;\!}F^{\:\!p\;\!\!,q\:\!;R}[u] : u
 \in C^{\:\!\infty}_{\rm c}(B_{\eta}^{}{\!}\bs{\;\!\!}\{0\}{\:\!\!})_{\rm rad}^{}{\!}
 \bs{\;\!\!}\{0\}\}.\label{symmetric2}
\end{align}
Here we remark  that
 the  functional $ F^{\:\!p\;\!\!,q\:\!;\:\!R}[u]$ itself  depends  on each $\eta>0$, nevertheless  the best constants $C^{\:\!p\;\!\!,q\:\!;\:\!R}$  and 
 $C^{\:\!p\;\!\!,q\:\!;R}_{\rm rad}$ are independent of $\eta>0$. To see this it  suffices  to employ a change  of  variables given by $x=\eta y$.
\begin{rem}\label{remark2.2} It  follows 
from 
Proposition \ref{classicalCKN} and Proposition \ref{classicalCKN2}
we have  interesting and mysterious relationships among the best constants  as follows:\medskip
\par\noindent
 If  $n\ge 2$, $1/p'\le \gamma_{p,q}$. $ R\ge R_{p,q}$ and $1/p-1/q\le 1/n$, then it holds  that  
\begin{equation}
S^{p,q;  1/p'}=S^{p,q;  1/p'}_{\rm rad}=
C^{\:\!p\;\!\!,q\:\!;R}_{\rm rad}
=
C^{\:\!p\;\!\!,q\:\!;R}. \label{relations}
\end{equation}
Here $\gamma_{p,q}$  and $R_{p,q}$ are positive numbers defined by (\ref{df2.4}) and (\ref{df2.6}) in Appendix.

\end{rem}
\medskip
\par
If  $n=1$,  then we  have the following  that is proved  for the sake of self-containedness.
\par\medskip
\begin{lem}\label{Lemma2.1}
 Assume  that $n=1$, $ 1<p\le q<\infty$ and  $\gamma\neq 0$.  Then we have  the followings:
\begin{enumerate}\item 
$S^{\:\!p\;\!\!,q\:\!;\:\!\gamma}(\R) =S^{\:\!p\;\!\!,q\:\!;\:\!\gamma}( (-\infty,0))=S^{\:\!p\;\!\!,q\:\!;\:\!\gamma}( (0,\infty))$.
\item
$S^{\:\!p\;\!\!,q\:\!;\:\!\gamma}_{\rm rad}(\R)= 2^{1-p/q}\;S^{\:\!p\;\!\!,q\:\!;\:\!\gamma}(\R)$.
\end{enumerate}
\end{lem}

\begin{proof}[Proof of Lemma \ref{Lemma2.1}]
For any $u\in C^\infty_c(\R\setminus\{0\})$, temporally we set 
$$ \|u\|_1(\vOmega) =\left(\int_{\vOmega}|u(t)|^p|x|^{p(1+\gamma)-1}\,dx \right)^{1/p}, \qquad 
\|u\|_2(\vOmega) =\left(\int_{\vOmega}|u(x)|^q |x|^{\gamma q-1}\, dx\right)^{1/q},$$
where $\vOmega= (-\infty,0)$ or $(0,\infty)$.
Since $1<p\le q<\infty$, we have $(1+t^p)^{1/p}\ge (1+t^q)^{1/q}$ for $t\ge 0$. Then we have
\begin{equation}\dstl{
 E^{p,q;\gamma}[u]^{1/p}
 \ge \min{\!}\left\{\:\!\!\dfrac{\|u'\|_1((-\infty,0))}
 {\|u\|_2((-\infty,0))},
 \dfrac{\|u'\|_1((0,\infty))}
 {\|u\|_2((0,\infty))}\right\}{\,\,\,}
 \mbox{ for }u\in C^{\:\!\infty}_{\rm c}(\R\bs {\;\!\!}\{0\}){\;\!\!}\bs{\;\!\!}\{0\}
}.\label{minmax}
\end{equation}
Thus we have 
$$ S^{p,q;\gamma}(\R)\ge \min( S^{p,q;\gamma}((-\infty,0)), S^{p,q;\gamma}((0,\infty)) ).$$
By  the symmetricity, we have $S^{p,q;\gamma}((-\infty,0))=S^{p,q;\gamma}((0,\infty)) $.
Since   the reverse inequality $S^{p,q;\gamma}(\R)\le S^{p,q;\gamma}((0,\infty)) $
is obvious  by  the definition (\ref{best}), the assertion (1) is proved.\par 
We proceed to the assertion (2). If  $u\in C^{\:\!\infty}_{\rm c}(\R\bs {\;\!\!}\{0\}){\;\!\!}\bs{\;\!\!}\{0\}$ is radial, then 
\begin{align*}&\|u'\|_1(\R)= 2^{1/p} \|u'\|_1((-\infty,0)) =2^{1/p} \|u'\|_1((0,\infty)),\\
& \|u\|_2(\R)=2^{1/q}\|u\|_2((-\infty,0)) =2^{1/q}\|u\|_2((0,\infty)). \end{align*}
Then   we have
$$ S^{\:\!p\;\!\!,q\:\!;\:\!\gamma}_{\rm rad}(\R)=  2^{1-p/q}\;S^{\:\!p\;\!\!,q\:\!;\:\!\gamma}((-\infty,0))
=2^{1-p/q}\;S^{\:\!p\;\!\!,q\:\!;\:\!\gamma}((0,\infty)).$$ Therefore the assertion (2) follows from the assertion (1).
\end{proof}


\subsection{A class of weight functions}
First we  introduce a class of weight functions which  is crucial in this paper.
\begin{df}\label{D0}
 Let us  set $\mathbf R_+= (0,\infty)$ and 
  \begin{equation}W(\mathbf R_+)=
\{ w\in C^1({\R}_+): w>0, \lim_{t\to+0}w(t)=a\, \text{  for  some }\,a\in [0,\infty] \}.\end{equation}

\end{df}
In  the  next  we define two subclasses of  this   large class.

\begin{df}\label{D1} Let  us  set 
 \begin{align} &P(\mathbf R_+)= \{ w\in W(\mathbf R_+) : \,  w^{-1} \notin L^1((0,\eta)) \, \text{ for some } \, \eta >0\}.\\
&
 Q(\mathbf R_+) =\{ w\in W(\mathbf R_+) :  \, w^{-1}\in L^1((0,\eta)) \, \text{ for any } \, \eta >0 \}.\end{align}
\end{df}

\begin{exam} 
\begin{enumerate}
\item 
$t^{\a } \in P(\mathbf R_+)$ if $\a \ge 1$  and 
$t^{\a} \in Q(\mathbf R_+)$ if $\a < 1$.
\item $e^{-1/t} \in  P(\mathbf R_+)$  and  $e^{1/t} \in Q(\mathbf R_+)$.

 \item For $ \a\in \mathbf R$, $ t^\a e^{-1/t}\in P(\mathbf R_+)$ and $t^\a e^{1/t}\in Q(\mathbf R_+)$.
\end{enumerate}
\end{exam}

\begin{rem} \begin{enumerate}
\item From  Definition \ref{D0} and Definition \ref{D1} it  follows   that $W(\mathbf R_+)= P(\mathbf R_+)\cup Q(\mathbf R_+)$ and  $P(\mathbf R_+)\cap Q(\mathbf R_+)=\phi.$
\item If  $w^{-1} \notin L^1((0,\eta))  \, \text{ for some } \, \eta >0$, then  $w^{-1} \notin L^1((0,\eta)) \, \text{ for any } \, \eta >0$. Similarly
if $w^{-1} \in L^1((0,\eta))  \, \text{ for some } \, \eta >0$, then  $w^{-1} \in L^1((0,\eta)) \, \text{ for any } \, \eta >0$.
\item
If  $w\in  P(\mathbf R_+)$, then $\dstl\lim_{t\to +0}w(t)=0$. Hence 
by  setting $w(0)=0$, $w$ is uniquely  extended to  a continuous function on $[0,\infty)$. 
On the other hand  if $w\in  Q(\mathbf R_+)$, then    possibly  $\dstl\lim_{t\to +0}w(t)=+\infty$.
\end{enumerate}
\end{rem}
\par\medskip
Lastly 
 we define  a function $f_{\eta}(t)$ in order to introduce  variants of the Hardy potentials:
\par\medskip
\begin{df}
\begin{equation}f_\eta(t)= \begin{cases} & \mu+ \int_t^\eta \frac{1}{w(s)}\,ds,\qquad  \text{ if }\quad
w\in P(\R_+),\\
&\int_0^t \frac{1}{w(s)}\,ds, \qquad\qquad \text{ if }\quad
w\in Q(\R_+).
\end{cases}\label{feta}
\end{equation}

\end{df}

\section{Main results }

\subsection{The $n$-dimensional CKN-type inequalities}

\par\bigskip
 Let  $B_\eta$ be  the ball $  \{x\in \R^n:  |x| <\eta\}$.
Define a strictly monotone    function $\varphi(\rho)\in C^1(0,\varphi^{-1} ({\eta}))$ as  follows:\par

\begin{df}\label{Definition3.1}
\begin{enumerate}\item
For  $w\in P(\R_+)$,  by $\varphi(\rho)$ we  denote the unique solution of  the integral equation

\begin{equation}
\rho^{-1}= \mu+ \int^\eta_{\varphi(\rho)}\frac{1}{w(s)}\,ds. \label{6.8}
\end{equation}

\item
For  $w\in Q(\R_+)$, by $\varphi(\rho)$ we  denote the unique solution of  the integral equation
\begin{equation}
\rho= \int _0^{\varphi(\rho)}\frac{1}{w(s)}\,ds.\label{6.9}
\end{equation}
\end{enumerate}
\end{df}
\medskip
First  we assume that $w\in P(\R_+)$, then
by differenciating (\ref{6.8}) we  have
\begin{equation}\varphi'(\rho)= \frac{w(\varphi(\rho))}{\rho^2},\quad
\lim_{\rho\to+0}\varphi(\rho)= 0, \quad \varphi(1/\mu)=\eta.\label{6.7}
\end{equation}
We  see that $\varphi$ is strictly monotone and  satisfies
 $ \varphi^{-1}(t)= {1}/{\left(\mu+ \int_t^\eta\frac1{w(s)}\,ds\right)}$.
\par \medskip
Secondly  we  assume that $w\in Q(\R_+)$, then 
 by differentiating (\ref{6.9}) we  have
\begin{equation}
 \varphi'(\rho)= w(\varphi(\rho)),\quad \lim_{\rho\to+0}\varphi(\rho)= 0.\label{6.12}
\end{equation}
Again $\varphi$ is strictly monotone  and  satisfies
 $ \varphi^{-1}(t)=  \int_0^t\frac1{w(s)}\,ds$.
 \par\medskip
Then we  define
\par\medskip
\begin{df} \label{Definition3.2}
\begin{enumerate}
\item
For $w\in P(\R_+)$, we  set  for $\rho\in (0, \varphi^{-1} ({\eta}))$
\begin{equation}
H(\rho)= \frac{\rho \varphi'(\rho)}{ \varphi(\rho)} = 
\frac{ w(\varphi(\rho))}{  \varphi(\rho)} \left( \mu+ \int^\eta_{\varphi(\rho)}w(s)^{-1}\,ds\right).\label{6.10}
\end{equation}
\item
For $w\in Q(\R_+)$ we  set for $\rho\in (0, \varphi^{-1} ({\eta}))$
\begin{equation}
H(\rho)= \frac{\rho \varphi'(\rho)}{ \varphi(\rho)} = 
\frac{  w(\varphi(\rho))}{ \varphi(\rho)}\int_0^{\varphi(\rho)} w(s)^{-1}\,ds. \label{6.13}
\end{equation}
\end{enumerate}
\end{df}
Now  we introduce {\bf the non-degenerate condition (NDC)} on $H(\rho)$ which assures  that
 $H(\rho)$ is  bounded away from $0$ as $\rho\to + 0$:  

\begin{df}$(${\bf the non-degenerate condition} $)$  Let $\eta>0$  and $w\in W(\R_+)$.  A weight function $w$ is said to satisfy   the non-degenerate condition (\ref{H-condition})  if 
\begin{equation}
C_0:=\dstl\inf_{0<\rho\le \varphi^{-1}(\eta)} H(\rho)>0.
%
\label{H-condition} \end{equation}
Here  
 $H(\rho)$ is defined  by Definition \ref{Definition3.2}.
\end{df}
Here we state the $n$-dimensional CKN-type inequalities, which is a natural extension of 
 the classical CKN-type inequalities (\ref{1.4}) and (\ref{CKNc}) in Section 2.

\begin{thm}\label{th4.2} Let $n\ge 1$, $1<p\le q <\infty$, $\mu>0$, $\eta>0$ and $ 0 \le 1/p -1/q \le1/n$. 
Assume that $w\in W(\mathbf R_+)$.
Moreover assume  that if $n>1$ and  $p<q$,  $H(\rho)$ satisfies the non-degenerate condition (\ref{H-condition}).
Then,  we have  the followings:
\begin{enumerate}\item
There exists a positive  number $C_n=C_n(p,q,\eta,\mu,w)$ such  that
 for any $u\in C_c^\infty(B_\eta\setminus \{0\})$  we  have
\begin{equation}
\int_{B_\eta} |\nabla u|^p w(|x|)^{p-1}|x|^{1-n}\,dx\ge  C_n \left(  \int_{B_\eta} 
\frac{ |u|^q |x|^{1-n}\,dx}{ w(|x|) f_\eta (|x|)^{1+q/{p'}}} \right)^{p/q}\label{6.16'}
\end{equation}
where $f_\eta(t)$ is  defined by (\ref{feta}).  

\item 
If $C_n$ is the best constant, then $C_n$ satisfies the followings:
\begin{enumerate}\item
If $n=1$, then $C_1=S^{p,q;1/p'}(\R)= 2^{p/q-1}S_{ \rm rad}^{p,q;1/p'}(\R)$.
\item 
If  $n>1$, then 
$C_n\ge\min(C_0^p,1)S^{p,q;1/p'}(\R^n)\:$ $(p<q)$; $C_n=(1/p')^p\:$ $(p=q)$. \par\noindent 
\item  If  $C_0\ge 1$ and  $1/p'\le \gamma_{p,q}$, then 
$$C_n=  S^{p,q;1/p'}(\R^n)=S_{\rm rad}^{p,q;1/p'}(\R^n).$$
\end{enumerate}
Here $C_0$ is  given by (\ref{H-condition}) and $\gamma_{p,q}$ is given by 
\begin{equation}\dstl{
 \gamma_{p\;\!\!,q}^{}=\dfrac{n-{\;\!\!}1}{1{\;\!\!}+q/p'\;\!\!}.
}
\end{equation}

\end{enumerate}
\end{thm}
\par\medskip
Theorem \ref{th4.2} will be  established in Section 4 by  virtue of  the variational principle and the  non-critical CKN-type inequalities.
%
We also  remark that
if $p=q$, then this  has been partially  treated  in \cite{ho3} as the Hardy-type inequality.
\par\medskip

Let  us consider the case where
 (NDC) is  violated, that is,  $ \dstl\inf_{0<\rho\le \varphi^{-1}(\eta)} H(\rho)=0$ holds. 
From  Definition \ref{Definition3.2} we  see that $H(\rho)>0 $ for $\rho\in (0, \varphi^{-1}(\eta))$,  hence   $\dstl\liminf_{\rho\to+0} H(\rho)=0$.
If we  assume  $\dstl{\lim_{\rho\to +0}H(\rho)= 0 }$ in stead, then the  inequality (\ref{6.16'}) fails to hold. Namely:
\par\medskip
\begin{thm}\label{Thm3.3}
 Let $n>1 $, $1<p< q <\infty$, $\mu>0$, $ 0 \le 1/p -1/q \le1/n$ and $R>0$. Assume that $w\in W(\mathbf R_+)$.
If $\dstl{\lim_{\rho\to +0}H(\rho)= 0 }$, then  the inequality (\ref{6.16'}) is not valid.
\end{thm}
Intuitively speaking, if
 $w \in P(\R_+)$  and $w$  vanishes   infinitely at  the origin or if
 $w \in Q(\R_+)$  and $w$ blows 
 up   infinitely at  the origin, then   (NDC) is violated.
 To explain more accurately, we introduce the following notion.
\par\medskip
\begin{df}\begin{enumerate}\item
For $w(t)\in P(\R_+)$,  $w(t)$  is  said to  vanish at  the origin in the infinite order,  if and only if
 for an arbitrary positive integer $m$ there exist  a positive number $t_m$  such  that $t_m\to 0 $ as $m\to\infty$ and 
\begin{equation} w(t)\le t^m,  \qquad t\in (0,t_m). \label{vanishinfinite}
\end{equation}
\item
For $w(t)\in Q(\R_+)$  $w(t)$ is  said to blow up at  the origin in the  infinite order,  if and only if
 for an arbitrary positive integer  $m$ there exists a positive $t_m$  such  that 
 $t_m\to 0 $ as $m\to\infty$ and
\begin{equation} w(t)\ge t^{-m},  \qquad t\in (0,t_m). \label{blowinfinite}
\end{equation}
\end{enumerate}
\end{df}
\par\medskip
Then we have  the following:
\par\medskip

\begin{thm}\label{Thm3.4} If
 $w \in P(\R_+)$  and $w$  vanishes at  the origin in  the  infinite order,  or if
 $w \in Q(\R_+)$  and $w$ blows 
 up  at  the origin  in the  infinite order, then   (NDC) is not satisfied.
\end{thm}
\par\medskip
To help understand the theorems, we give typical examples: 
\par\medskip
\begin{exam}
\begin{enumerate}\item When either $w(t)=e^{-1/t} \in P(\R_+)$ or $w(t)=e^{1/t}\in Q(\R_+)$, \par $H(\varphi^{-1}(t))= O(t)$ as $t\to +0$, or equivalently $ H(\rho)= O(\varphi(\rho))$ as $\rho=\varphi^{-1}(t)\to +0$. 
\item Moreover,
if $w(t)= \exp({\pm t^{-\a}}) $ with $\a> 0$, then $H(\varphi^{-1}(t))= O(t^{\a})$ as $t\to +0$. 
{ In fact, it  holds  that $\dstl\lim_{t\to +0} H(\varphi^{-1}(t))/t^{\a}=1/\a$. }

\end{enumerate}
\end{exam}

\par\medskip
\begin{exam}
Let $n\ge 1$, $1<p\le q <\infty$,  $ 0 \le 1/p -1/q \le1/n$ and $0<\eta$.
\par\noindent  Let $w(t)= t^{p'\gamma+1}$. 
\par
If $\gamma> 0$, then $w(t)\in P(\R_+)$ and we  have for $\mu=  \frac1{p'\gamma}\eta^{-p'\gamma}$,
$$ f_\eta (t) =\frac1{p'\gamma}t^{-p'\gamma},\quad
\frac1\rho=\frac1{p'\gamma}\varphi(\rho)^{-p'\gamma}\quad\text{and}\quad  H(\rho)=\frac{\rho \varphi'(\rho)}{\varphi(\rho)}=\frac{1}{p'\gamma}.$$
If $\gamma=0 $, then $w(t)\in P(\R_+)$ and we  have for $\mu>0$,
$$ f_\eta(t)=  \mu+  \log \eta/t, \quad \frac1\rho=  \mu+  \log \eta/{\varphi(\rho)}\quad\text{and}\quad  H(\rho)=\frac 1\rho\: (\ge \mu).$$
If $\gamma <0$, then $w(t)\in Q(\R_+)$ 
$$ f_\eta(t)  =-\frac1{p'\gamma}t^{-p'\gamma},\quad
\rho=-\frac1{p'\gamma}\varphi(\rho)^{-p'\gamma}\quad\text{and}\quad
H(\rho)=\frac{\rho \varphi'(\rho)}{\varphi(\rho)}=-\frac{1}{p'\gamma}.  $$
Thus  we have 
$$ C_0=\begin{cases} &  (p'|\gamma|)^{-1},\qquad \text{ if } \; \gamma\neq 0,\\
& \mu,\quad \qquad\qquad \text{ if }\; \gamma =0.
\end{cases}$$
Assume that either $0< |\gamma|\le 1/p'$ or $\gamma=0, \,\mu\ge 1$. 
Then we  have $H(\rho) \ge 1 $ in $ (0,\varphi^{-1}(\eta))$.
\par\noindent From Theorem \ref{th4.2}, noting that
 $w(|x|)^{p-1}|x|^{1-n} = |x|^{p(1+\gamma)-n}$, we  have 
\begin{equation}
\int_{B_\eta} |\nabla u|^p |x|^{p(1+\gamma)-n}\,dx\ge  C \left(  \int_{B_\eta} 
\frac{ |u|^q |x|^{1-n}\,dx}{ w(|x|) f_\eta (|x|)^{1+q/{p'}}} \right)^{p/q}\label{6.16}
\end{equation}
where  $C=\min(C_0^p,1)S^{p,q;1/p'}= S^{p,q;1/p'}$
\end{exam}
We give a  non-doubling weight function $w(t)\in W(\R_+)$ for  which the condition (NDC) is satisfied.
\begin{exam}\label{example3.3}
Let $1<\eta$ and for  simplicity we make examples with $(0,\eta]$ instead of $R_+$.
First we give an example in $P((0,\eta])$. (For the  definition of $P((0,\eta])$ and $Q((0,\eta])$, see Remark \ref{remark3.1}.)
Define
$$ z(t)= \frac1{\log(e\eta/t)}, \qquad  0<t\le\eta.$$
We take a function  $z_1(t) \in C^1((0,\eta])$ such  that $z_1(t)$  satisfies  the estimate
$$z(t)\le z_1(t)\le 1$$  and the condition  
\begin{equation} \begin{cases}z_1(t_{2m})&= z(t_{2m}),\\
z_1(t_{2m+1})&=1,
\end{cases} \qquad (m=1,2,\cdots), \label{3.12}\end{equation}
where   $t_m=1/m\; (m=1,2,\cdots).$
Now  we  set $$w(t)= t z_1(t)\in C^1((0,\eta]).$$ Since $w(t) =t z_1(t)\le t$ holds, we  have $\int_t^\eta w(s)^{-1}\,ds\ge \int_t^\eta 1/s\,ds= \log(\eta/t)\to \infty \;(t\to +0)$.
Therefore  $w(t)\in P((0,\eta])$. 
We  show  that $w(t) $ becomes a non-doubling weight  function. In fact, if $m$ is an  odd number,  then we  have 
$$ \frac{w(t_{2m})}{w(2t_{2m})}= \frac{w(t_{2m})}{w(t_{m})}=\frac12\frac{1}{\log(2me\eta)}\to 0 \text{  as  } m\to\infty.$$
On  the other hand,  for $\rho=\varphi^{-1}(t) \; (0<t\le\eta)$,
\begin{align*} H(\varphi^{-1}(t))&= \frac{w(t)}{t}\left(\mu +\int_t^\eta \frac1{w(s)}\,ds \right)
= z_1(t)\left(\mu +\int_t^\eta \frac1{sz_1(s)}\,ds \right)\\
&\ge  z(t)\left(\mu+ \int_t^\eta \frac1{s}\,ds\right)= \frac1{\log(e\eta/t)}\cdot (\mu+\log(\eta/t))\ge 1\quad \mbox{ for }\; \mu\ge 1.
\end{align*}
Hence the condition (NDC) is satisfied. \par 
In  the next  we give  an example in $Q((0,\eta])$.
We  set $$z(t)=\log(e\eta/t)\qquad  0<t\le \eta.$$ Take a function $z_1(t)  \in C^1((0,\eta])$ 
satisfying  the condition (\ref{3.12})  and 
the estimate
$$ 1\le z_1(t)\le z(t).$$ 
 We define $$w(t)=tz(t)^2 z_1(t)\in C^1((0,\eta]).$$ 
Since $w(t) =t z(t)^2z_1(t)\ge t z(t)^2$ holds, we  have $\int_0^t w(s)^{-1}\,ds\le \int_0^\eta1/(s(\log(e\eta/s))^2)\,ds= 1$.
Therefore  $w(t)\in Q((0,\eta])$. 
We  show  that $w(t) $ becomes a non-doubling weight function. If $m$ is an  odd number,  then we  have 
$$ \frac{w(t_{2m})}{w(2t_{2m})}= \frac{w(t_{2m})}{w(t_{m})}
=\frac12\frac{(\log(2me\eta))^3}{(\log(me\eta))^2}\to \infty \text{  as  } m\to\infty.$$

On  the other hand,  for $\rho=\varphi^{-1}(t), \; 0<t\le\eta$,
\begin{align*} H(\varphi^{-1}(t))&= \frac{w(t)}{t}\int_0^t \frac1{w(s)}\,ds 
= z(t)^2z_1(t)\int_0^t \frac1{s z(s)^2z_1(s)}\,ds \\
&\ge  z(t)^2 \int_0^t \frac1{sz(s)^3}\,ds=\frac{ (\log(e\eta/t))^2\cdot (\log(e\eta/t))^{-2}}2= \frac12.
\end{align*}
Hence the condition (NDC) is satisfied. \par

\end{exam}

\begin{rem} \label{remark3.1}\begin{enumerate}\item
In Example \ref{example3.3}, $P((0,\eta])$ and  $Q((0,\eta])$ are defined in a similar way by (1.1) with $\R_+$ replaced by $(0,\eta]$.
\item Theorem \ref{th4.2}, Theorem \ref{Thm3.3} and Theorem \ref{Thm3.4} will be established in Section 4, Section 5 and Section 6 respectively.
\item When $p=1$, one  can obtain   similar results as  in this  section under a suitable  modification. See \cite{ho4}.

\end{enumerate}
\end{rem}

\section{Proof of Theorem \ref{th4.2}}

\begin{proof}[Proof of Theorem \ref{th4.2}]
We  employ a polar coordinate system $x= r\omega $ for  $r= |x|$ and $\omega\in S^{n-1}$.
By $\Lap_{\Sn}^{}$  we  denote  the Laplace-Beltrami operator on a unit sphere $\dstl{
 S^{\:\!n\:\!-1}}$. 
 Then  a gradient operator $\vLambda$ on 
 $\dstl{
 S^{\:\!n\:\!-1}
}$ is defined by  
\begin{equation}\dstl{
 {\int_{\;\!\!\Sn}}
 (-{\;\!}\Lap_{\Sn}^{}\xi_1){\:\!}\xi_2{\;\!}dS
 ={\int_{\;\!\!\Sn}}\vLambda{\:\!}\xi_1{\cdot}\vLambda{\:\!}\xi_2{\;\!}dS
 {\,\,\,}\mbox{ for }\xi_1,\xi_2\in C^{\:\!2}(S^{\:\!n\:\!-1})
}.
\end{equation}
Here we have 
\begin{equation}\dstl{
 \Lap{\:\!}\xi_1=\dfrac{1}{r^{\:\!n-1}\;\!\!}
{\partial_r}\big(r^{\:\!n-1}{\partial_r{\:\!}\xi_1}\big)
 +{\;\!\!}\dfrac{1}{r^{2}\;\!\!}\Lap_{\Sn}^{}\xi_1,{\,\,\,}
 |\nabla \xi_1|^{2}=\big|{\partial_r\xi_1}\big|^{2}{\;\!\!}
 +{\;\!\!}\dfrac{1}{r^{2}\;\!\!}|\vLambda{\:\!}\xi_1|^{2}
},
\end{equation}
where
\begin{equation}\dstl{
{\partial_r\xi_1}(x)=\dfrac{x}{|x|}{\cdot}\nabla \xi_1(x)
}.
\end{equation}
\par\medskip
Then, the inequality  (\ref{6.16'}) is  transformed to the  following:
\begin{equation}
\begin{split}
\int_{S^{n-1}}& \,dS \int _0^\eta  \left( (\partial_r u)^2+ \frac{(\vLambda u)^2}{r^2}\right)^{p/2} w(r)^{p-1}\,dr\\
&\ge  C_n \left( \int_{S^{n-1}} \,dS\int _0^\eta
\frac{ |u|^q \,dr}{ w(r) f_\eta (r)^{1+q/{p'}}} \right)^{p/q},\quad( u(x)= u(r\omega)\in C_c^\infty(B_\eta\setminus\{0\}) ).
\end{split}\label{5.13}
\end{equation}
Here  $dS$  denotes  surface elements on  $S^{n-1}$ if $n>1$ and a counting measure on $ S^0 =\{-1,1\} $ if $n=1$.
Here we remark  that if $n=1$, then  $\Lap_{\Sn}=\vLambda=0$.
\par\medskip\noindent
{ Proof of (1).} 
First  we consider the case where $w\in P(\R_+)$, and
 we  set  $y= \rho \omega$ ($\rho=|y|,\;r=\varphi(\rho)$) and $U(y)= U(\rho\omega) =u(\varphi(\rho)\omega)$, then 
$U(y)\in C_c^\infty(B_{\tilde{\eta}}\setminus \{0\})$ with $\tilde{\eta}= \varphi^{-1}(\eta)\,(=1/\mu)$.
Then, the inequality (\ref{5.13}) is transformed to the following: 
\begin{equation}
\begin{split}
\int_{S^{n-1}} \,dS &\int _0^{\tilde \eta}\left( (\partial_\rho U)^2+H(\rho)^2 \frac{(\vLambda U)^2}{\rho^2}\right)^{p/2} \rho^{2(p-1)}\,d\rho \label{6.10}\\
&\ge  C_n \left( \int_{S^{n-1}} \,dS\int _0^{\tilde \eta}
 |U|^q   \rho^{-1+q/{p'}} \,d\rho\right)^{p/q}, 
\end{split}
\end{equation}
or equivalently 
 \begin{align}
\int _{B_{\tilde \eta}}\left( (\partial_\rho U)^2+H(\rho)^2 \frac{(\vLambda U)^2}{\rho^2}\right)^{p/2} \rho^{2p-1-n}\,dy
\ge  C_n \left( \int_{B_{\tilde \eta}} 
 |U|^q   \rho^{q/{p'} -n} \,dy\right)^{p/q}. \label{6.11'}
\end{align}
In order  to show  the inequality (\ref{6.16'}), it suffices to establish (\ref{6.11'})
for  some positive number $C_n$. If $p=q$ holds, then (\ref{6.11'}) follows direct from the classical Hardy inequality, hence  we assume  that $p<q$. 
From (\ref{H-condition}),  the left hand side of (\ref{6.11'}) is estimated from below in the following way.

 \begin{align*}
\int _{B_{\tilde \eta}} \left( (\partial_\rho U)^2+H(\rho)^2 \frac{(\vLambda U)^2}{\rho^2}\right)^{p/2}& \rho^{2p-1-n}\,dy
\ge  
\int _{B_{\tilde \eta}}\left( (\partial_\rho U)^2+C_0^2 \frac{(\vLambda U)^2}{\rho^2}\right)^{p/2} \rho^{2p-1-n}\,dy\\
&\ge \min(C_0^p,1)
\int _{B_{\tilde \eta}}\left( (\partial_\rho U)^2+\frac{(\vLambda U)^2}{\rho^2}\right)^{p/2} \rho^{2p-1-n}\,dy\\ &= \min(C_0^p,1)
\int _{B_{\tilde \eta}}|\nabla_y U|^p\rho^{p(1+1/p')-n}\,dy
\\
&\ge \min(C_0^p,1)S^{p,q;1/p'} \left( \int_{B_{\tilde \eta}} 
 |U|^q   \rho^{q/{p'} -n} \,dy\right)^{p/q}. \label{}
\end{align*}
In the last step we used the CKN type inequality (\ref{1.4}) with $\gamma=1/p'$. This  proves the assertion  (\ref{6.11'}) with $C_n\ge\min(C_0^p,1)S^{p,q;1/p'} $.
\par
Secondly  we  assume  that $w\in Q(\R_+)$.
By putting $r=\varphi(\rho)$,  in a similar way 
the inequality (\ref{5.13}) is transformed to the following: 
\begin{equation}
\begin{split}
\int_{S^{n-1}} \,dS &\int _0^{\tilde \eta}\left( (\partial_\rho U)^2+H(\rho)^2 \frac{(\vLambda U)^2}{\rho^2}\right)^{p/2}\,d\rho  \label{6.14}\\
&\ge  C_n \left( \int_{S^{n-1}} \,dS\int _0^{\tilde \eta}
 |U|^q   \rho^{-1-q/{p'}} \,d\rho\right)^{p/q},
\end{split}
\end{equation}
or equivalently 
 \begin{align}
\int _{B_{\tilde \eta}}\left( (\partial_\rho U)^2+H(\rho)^2 \frac{(\vLambda U)^2}{\rho^2}\right)^{p/2} \rho^{1-n}\,dy
\ge  C_n \left( \int_{B_{\tilde \eta}} 
 |U|^q   \rho^{-q/{p'} -n} \,dy\right)^{p/q}. \label{6.15}
\end{align}
We assume  that $p<q$ by  the same reason in the  previous step. 
Then we  have 

 \begin{align*}
\int _{B_{\tilde \eta}} \left( (\partial_\rho U)^2+H(\rho)^2 \frac{(\vLambda U)^2}{\rho^2}\right)^{p/2}& \rho^{1-n}\,dy
\ge  
\int _{B_{\tilde \eta}}\left( (\partial_\rho U)^2+C_0^2 \frac{(\vLambda U)^2}{\rho^2}\right)^{p/2} \rho^{1-n}\,dy\\
&\ge \min(C_0^p,1)
\int _{B_{\tilde \eta}}\left( (\partial_\rho U)^2+\frac{(\vLambda U)^2}{\rho^2}\right)^{p/2} \rho^{1-n}\,dy\\ &= \min(C_0^p,1)
\int _{B_{\tilde \eta}}|\nabla_y U|^p\rho^{p(1-1/p')-n}\,dy
\\
&\ge \min(C_0^p,1)S^{p,q;-1/p'} \left( \int_{B_{\tilde \eta}} 
 |U|^q   \rho^{-q/{p'} -n} \,dy\right)^{p/q}. \label{}
\end{align*}
In the last step we used the CKN type inequality (\ref{1.4}) with $\gamma=-1/p'$. This  proves the assertion  (\ref{6.15}) with $C_n\ge\min(C_0^p,1)S^{p,q;-1/p'}= \min(C_0^p,1)S^{p,q;1/p'}$. 
\par\medskip\noindent
{\it Proof of (2).} (a). Assume that  $n=1$. Then (\ref{6.16'}) is equivalent to the following: 
For $U\in C_c^\infty(B_{\tilde{\eta}}\setminus \{0\})$ with $\tilde{\eta}= \varphi^{-1}(\eta)$, 
 \begin{equation}\label{4.9}
 \begin{cases}&
\int _{B_{\tilde \eta}} |\partial_\rho U|^p \rho^{2(p-1)}\,dy
\ge  C_1 \left( \int_{B_{\tilde \eta}} 
 |U|^q   \rho^{q/{p'} -1} \,dy\right)^{p/q}, \qquad w\in P(R_+), \\
&\int _{B_{\tilde \eta}}|\partial_\rho U|^p \,dy
\ge  C_1 \left( \int_{B_{\tilde \eta}} 
 |U|^q   \rho^{-q/{p'} -1} \,dy\right)^{p/q},\qquad\qquad  w\in Q(R_+). 
 \end{cases}
 \end{equation}
Then it follows from the  non-critical CKN-type inequalities (\ref{1.4}) with $n=1$  that  we  have $C_1= S^{p,q; 1/p'}(B_{\tilde \eta})$ $= S^{p,q; 1/p'}(\R^n)$. As for  the last equality, see the remark just after Definition \ref{df2.2}.
\par\noindent (b) Assume  that $n>1$. If $p>q$, then  from the argument in (1) 
$C_n\ge \min(C_0^p,1)S^{p,q;1/p'}(\R^n)$ holds. If $p=q$, then the inequalities (\ref{6.10}) and (\ref{6.14}) are reduced to the Hardy-type inequalities, hence we  have $C_n=S^{p,p; 1/p'}_{\rm rad}=(1/p')^p$. (See Proposition\ref{relation}, (1).)
\par\noindent(c)
Assume  that $C_0\ge 1$ and $1/p'\le \gamma_{p,q}=(n-1)/(1+q/p')$. By 
Proposition \ref{classicalCKN} together with
Proposition \ref{relation}, $S^{p,q; 1/p'}=S^{p,q;-1/p'}=S^{p,q; 1/p'}_{\rm rad}=S^{p,q;-1/p'}_{\rm rad}$ holds.
Then one can assume  $U\in C_c^\infty(B_{\tilde{\eta}}\setminus \{0\})_{\rm rad}$ so that we have $\vLambda U\equiv 0$. Therefore the assertion is now clear.

\end{proof}

\begin{rem}
We note  that if $1<p\le (n+1)/2$, then  $1/p'\le \gamma_{p,q}$ for any $q \in (p, np/(n-p)]$.
\end{rem}
\medskip

\section{Proof of Theorem \ref{Thm3.3}}

\begin{proof}
{\bf Step 1.}
We assume  that $H(\rho)\equiv 0$ in $(0,\tilde \eta)$ with $\varphi(\tilde \eta)=\eta$.
First we assume  that $w\in P(\R_+)$. Then
It follows from (\ref{6.10}) that the desired inequality (\ref{6.16}) has  the form

\begin{align}
\int_{S^{n-1}} \,dS \int _0^{\tilde \eta} |\partial_\rho U|^p \rho^{2(p-1)}\,d\rho 
\ge  C \left( \int_{S^{n-1}} \,dS\int _0^{\tilde \eta}
 |U|^q   \rho^{-1+q/{p'}} \,d\rho\right)^{p/q}. \label{6.18}
\end{align}
By $ M(S^{n-1})$ we denote a set of all measurable functions on $S^{n-1}$.
For $A(\rho) \in C^\infty_c((0,\tilde \eta))\setminus \{0\}$  and  $ B(\omega) \in  M(S^{n-1})$, 
we define a test function $U(\rho \omega)=A(\rho) \cdot B(\omega) $ such  that
$$ 0\le B(\omega) \in L^p(S^{n-1}), \quad \text {but}\quad B(\omega) \notin L^q(S^{n-1}).$$
Then, in the inequality (\ref{6.18}),  we see that
\begin{equation}
 \begin{cases}\mbox{(LHS) } &=\int_{S^{n-1} } |B(\omega)|^p \,dS\int_0^{\tilde \eta} |\partial_\rho A(\rho)|^p\rho^{2(p-1)}
\,d\rho<\infty,\\
 \mbox{(RHS) }& =\left(\int_{S^{n-1} } |B(\omega)|^q \,dS\right)^{p/q}
\left(\int_0^{\tilde \eta} | A(\rho)|^q\rho^{q/{p'} -1}
\,d\rho\right)^{p/q}=\infty. 
\end{cases} \label{6.19}
\end{equation}
Here  (LHS)  and (RHS) are abbreviations of the left-hand side and the right-hand side  respectively.
Therefore the  inequality (\ref{6.18}) does not hold, and this proves the assertion provided that $H\equiv0$.
Secondly we assume  that  $w\in Q(\R_+)$.
Then
It follows from (\ref{6.14}) that the desired inequality (\ref{6.16}) has  the form
\begin{align}
\int_{S^{n-1}} \,dS \int _0^{\tilde \eta} |\partial_\rho U|^p\,d\rho 
\ge  C \left( \int_{S^{n-1}} \,dS\int _0^{\tilde \eta}
 |U|^q   \rho^{-1-q/{p'}} \,d\rho\right)^{p/q}. \label{6.20}
\end{align}
Again using the same test function $U=A(\rho)B(\omega)$, we see that
\begin{equation}
 \begin{cases}\mbox{(LHS) } &=\int_{S^{n-1} } |B(\omega)|^p \,dS\int_0^{\tilde \eta} |\partial_\rho A(\rho)|^p
\,d\rho<\infty,\\
 \mbox{(RHS) }& =\left(\int_{S^{n-1} } |B(\omega)|^q \,dS\right)^{p/q}
\left(\int_0^{\tilde \eta} | A(\rho)|^q\rho^{-q/{p'} -1}
\,d\rho\right)^{p/q}=\infty. 
\end{cases}
\end{equation}
Therefore the  inequality (\ref{6.20}) does not hold, and we have  the desired assertion.
\begin{rem} From Theorem \ref{th4.2}, Lemma \ref{Lemma2.1} and  (\ref{4.9})  we  have the  followings:
  \begin{enumerate}
\item For  $w(\rho)=\rho^2$, we have
\begin{equation}
\int_0^{\tilde \eta}  |\partial_\rho A(\rho)|^p\rho^{2(p-1)}
\,d\rho\ge  C \left(\int_0^{\tilde \eta} | A(\rho)|^q\rho^{q/{p'} -1}
\,d\rho\right)^{p/q},
\end{equation}
\item For $w(\rho)=1$ we  have
\begin{equation}
\int_0^{\tilde \eta}  |\partial_\rho A(\rho)|^p
\,d\rho\ge  C \left(\int_0^{\tilde \eta} | A(\rho)|^q\rho^{-q/{p'} -1}
\,d\rho\right)^{p/q},
\end{equation}
\end{enumerate}

where $C$ is a positive number independent of  each function  $A(\rho)\in C^\infty_c((0,\tilde \eta))$.
\end{rem}
\par\medskip\noindent
{\bf Step 2.} 
We assume  that $\dstl\lim_{\rho\to +0} H(\rho)=0$ and $w\in P(\R_+)$.
As in the previous  step, we take  $B(\omega)\in L^p(S^{n-1})$ with 
$B(\omega) \notin L^q(S^{n-1})$. 
Let $B_j(\omega)$ be a mollification of $B$ such  that $B_j(\omega)\in C^\infty(S^{n-1})$,
$B_j \to B  $ in $L^p(S^{n-1})$ but
\begin{equation}
\int_{S^{n-1}}|B_j(\omega)|^q\,dS \to \infty \quad (j\to\infty).\label{infty}
\end{equation}
 Let $\{ \e_j\}$ be  a sequence of  numbers  such that $0<\e_j<1$,    $\e_j\to 0$ as $j\to\infty$ 
and 
\begin{equation}
H(\rho)^{p}\cdot\int _{S^{n-1}} |\vLambda B_j(\omega)|^p \,dS \le 1\qquad ( 0<\rho\le \e_j \tilde\eta,  \, j=1,2,3,\ldots).\label{6.25}
\end{equation}
We take and fix  an $ A(\rho)\in  C^\infty_c((0,  \tilde \eta))\setminus\{0\}$ satisfying  
\begin{equation} \int_0^{\tilde \eta} |\partial_\rho A(\rho)|^p\rho^{2(p-1)}
\,d\rho =1. \label{6.26} \end{equation}
Define 
\begin{equation}A_j(\rho)= \e_j^{-1/{p'}} A(\rho/{\e_j})\qquad (j=1,2,3,\ldots ). \label{Aj}\end{equation}
Then,  for $j=1,2,3,\ldots$ we see that
$ A_j(\rho)\in  C^\infty_c(0,\e_j \tilde \eta) \subset C^\infty_c(0,\tilde \eta) $ and
\begin{align}
 &\int_0^{\tilde \eta \e_j} |\partial_\rho A_j(\rho)|^p\rho^{2(p-1)}
\,d\rho =\int_0^{\tilde \eta}  |\partial_\rho A(\rho)|^p\rho^{2(p-1)}
\,d\rho =1,\label{6.28}\\ 
&\int_0^{\tilde \eta\e_j}  |A_j(\rho)|^q\rho^{-1+q/p'}
\,d\rho=\int_0^{\tilde \eta} | A(\rho)|^q\rho^{-1+q/p'}
\,d\rho.
\label{6.29} \end{align} 
Then we define a sequence of test functions $U_j= A_j(\rho)\cdot B_j(\omega) \in  C^\infty_c((0, \e_j \tilde \eta))\times C^\infty(S^{n-1})$.
If we  show the following properties, then the assertion clearly follows:
\begin{enumerate}
\item
\begin{equation}\label{6.30}
\int_{S^{n-1}} \,dS \int _0^{\tilde \eta\e_j}\left( (\partial_\rho U_j)^2+H(\rho)^2 \frac{(\vLambda U_j)^2}{\rho^2}\right)^{p/2} \rho^{2(p-1)}\,d\rho <\infty,
\end{equation}
\item
\begin{equation}\label{6.31}
 \left( \int_{S^{n-1}} \,dS\int _0^{\tilde \eta\e_j}
 |U_j|^q   \rho^{-1+q/{p'}} \,d\rho\right)^{p/q} \to \infty \quad \mbox{as }\quad j\to\infty.
\end{equation}
\end{enumerate}
It follows from (\ref{infty}) and (\ref{6.29}) we have (\ref{6.31}), and  hence
 it suffices to show (\ref{6.30}).
By  the definition  we immediately see that
\begin{equation}\label{6.32}
 \int_0^{\tilde \eta\e_j} |\partial_\rho A_j(\rho)|^p\rho^{2(p-1)}
\,d\rho \int_{S^{n-1} } |B_j(\omega)|^p \,dS= \int_{S^{n-1} } |B_j(\omega)|^p \,dS<\infty. \end{equation} 
By  Hardy's inequality, for some positive number $C$ we have 
\begin{equation}
\int_0^{\tilde \eta \e_j} |\partial_\rho A_j(\rho)|^p\rho^{2(p-1)}
\,d\rho\ge C \int_0^{\tilde \eta\e_j} |A_j(\rho)|^p  \rho^{p-2}\,d\rho, \qquad (j=1,2,3,\ldots ).\label{6.33}
\end{equation}
Then  we have
\begin{equation}
\begin{split}
\int_0^{\tilde \eta\e_j} |A_j(\rho)|^p & \frac{H(\rho)^p}{\rho^p}\rho^{2(p-1)}\,d\rho \int_{S^{n-1}} |\vLambda B_j(\omega)|^p\, dS\\
&=\int_0^{\tilde \eta\e_j} |A_j(\rho)|^p  \rho^{p-2} H(\rho)^p\,d\rho \int_{S^{n-1}} |\vLambda B_j(\omega)|^p\, dS \\
& \le C^{-1}\int_0^{\tilde \eta\e_j} |\partial_\rho A_j(\rho)|^p\rho^{2(p-1)}
\,d\rho
<\infty\qquad  ( (\ref{6.25}), (\ref{6.33}) )
\end{split}\label{6.34}
\end{equation}
Since $(a^2+b^2)^{p/2}\le 2^{p/2}(a^p+b^p),\,(a,b\ge 0)$, we have (\ref{6.26}), hence the assertion is proved.
\par\medskip
 Secondly we  assume that $w\in Q(\R_+)$. 
Let $B_j(\omega)\in C^\infty(S^{n-1}) \, (j=1,2,3,\ldots)$ be the same function as before. 
We take   an $ A(\rho)\in  C^\infty_c((0,  \tilde \eta))\setminus\{0\}$ satisfying  
\begin{equation} \int_0^{\tilde \eta} |\partial_\rho A(\rho)|^p
\,d\rho =1. \label{6.35} \end{equation}
Define 
\begin{equation}A_j(\rho)= \e_j^{1/{p'}} A(\rho/{\e_j})\qquad (j=1,2,3,\ldots ). \label{Aj2}\end{equation}
Then,  for $j=1,2,3,\ldots$ we see that
$ A_j(\rho)\in  C^\infty_c(0,\e_j \tilde \eta)\subset  C^\infty_c((0,\tilde \eta)) $ and
\begin{align}
 &\int_0^{\tilde \eta \e_j} |\partial_\rho A_j(\rho)|^p
\,d\rho =\int_0^{\tilde \eta}  |\partial_\rho A(\rho)|^p
\,d\rho =1,\label{6.37}\\ 
&\int_0^{\tilde \eta\e_j}  |A_j(\rho)|^q\rho^{-1-q/p'}
\,d\rho=\int_0^{\tilde \eta} | A(\rho)|^q\rho^{-1-q/p'}
\,d\rho.
\label{6.38} \end{align} 
Now we define a sequence of test functions $U_j= A_j(\rho)\cdot B_j(\omega) \in  C^\infty_c((0, \e_j\tilde \eta))\times C^\infty(S^{n-1})$.
If we can show the following properties, then the assertion  follows in a similar way:
\begin{enumerate}
\item
\begin{equation}\label{6.39}
\int_{S^{n-1}} \,dS \int _0^{\tilde \eta\e_j}\left( (\partial_\rho U_j)^2+H(\rho)^2 \frac{(\vLambda U_j)^2}{\rho^2}\right)^{p/2} \,d\rho <\infty,
\end{equation}
\item
\begin{equation}\label{6.40}
 \left( \int_{S^{n-1}} \,dS\int _0^{\tilde \eta\e_j}
 |U_j|^q   \rho^{-1-q/{p'}} \,d\rho\right)^{p/q} \to \infty \quad \mbox{as }\quad j\to\infty.
\end{equation}
\end{enumerate}
Since (\ref{6.40}) follows direct from (\ref{infty}) and  (\ref{6.38}), it suffices to show (\ref{6.39}).
By  the definition  we immediately see that
\begin{equation}\label{6.41}
 \int_0^{\tilde \eta\e_j} |\partial_\rho A_j(\rho)|^p
\,d\rho \int_{S^{n-1} } |B_j(\omega)|^p \,dS= \int_{S^{n-1} } |B_j(\omega)|^p \,dS<\infty. \end{equation} 
By  Hardy's inequality, for some positive number $C$ we have 
\begin{equation}
\int_0^{\tilde \eta \e_j} |\partial_\rho A_j(\rho)|^p
\,d\rho\ge C \int_0^{\tilde R\e_j} |A_j(\rho)|^p  \rho^{-p}\,d\rho, \qquad (j=1,2,3,\ldots ).\label{6.42}
\end{equation}
Then  we have
\begin{equation}
\begin{split}
\int_0^{\tilde \eta\e_j} |A_j(\rho)|^p & \frac{H(\rho)^p}{\rho^p}\,d\rho \int_{S^{n-1}} |\vLambda B_j(\omega)|^p\, dS\\
& \le C^{-1}\int_0^{\tilde \eta\e_j} |\partial_\rho A_j(\rho)|^p
\,d\rho
<\infty\qquad  ( (\ref{6.25}), (\ref{6.42}) )
\end{split}\label{6.43}
\end{equation}
Then we have (\ref{6.39}) as before, hence the assertion is proved.
\end{proof}

\section{Proof of Theorem \ref{Thm3.4}}
\begin{proof}

First we treat the  case that $w\in P(\R_+)$ and $w$ vanishes in infinite order at  the origin.
Namely  we assume that for an arbitrary positive number $m$ there exists a positive $t_m$  such  that $t_m\to 0 $ as $m\to\infty$ and 
\begin{equation} w(t)\le t^m,  \qquad t\in (0,t_m). \label{6.46}
\end{equation}
 Since $\varphi(\rho)\; (\varphi(0)=0)$ is  increasing, 
for $\eta>0$ we  set $\tilde \eta=\varphi^{-1}(\eta)$.
Now we assume  on the contrary that for  some  positive number $C_0$, 
$$H(\rho) \ge C_0, \qquad 0< \rho\le \tilde\eta.$$
Then  $C_0/\rho\le  \varphi'(\rho)/\varphi(\rho)$ holds for $\rho\in (0,\tilde\eta]$, hence by integrating the  both side over an interval $[\rho, \tilde\eta]$ we have 
\begin{equation} 
\varphi(\rho) \le \frac{\varphi(\tilde\eta)}{\tilde\eta^{C_0}} \rho^{C_0}, \qquad \rho \in (0,\tilde\eta). \label{6.47}
\end{equation}

Since $ H(\rho)= \rho\varphi'(\rho)/\varphi(\rho)= w(\varphi(\rho))/(\rho \varphi(\rho))$ holds, by (\ref{6.46}) we  have $ C_0 \le {\varphi(\rho)^{m-1}}/{\rho}$ for sufficiently small $\rho>0$,  more precisely
\begin{equation} (C_0\rho)^{1/(m-1)} \le \varphi(\rho),  \qquad \rho\in (0, \varphi^{-1}(t_m)). \label{6.48}
\end{equation}
Then we have $$ 1\le \frac{\varphi(\tilde\eta)}{C_0^{1/(m-1)}{\tilde\eta}^{C_0}} \rho^{{C_0}-1/(m-1)} 
\qquad \rho\in (0, \varphi^{-1}(t_m)).$$
If $m$ is sufficiently large, then this does not hold, hence
 the assertion is proved by  a contradiction.
\par\medskip
In the next  we  treat the case that $w\in Q(\R_+)$ and $w$ blows up at  the origin in infinite order. Then we assume  that
 for an arbitrary positive number $m$ there exists a positive $t_m$  such  that 
 $t_m\to 0 $ as $m\to\infty$ and
\begin{equation} w(t)\ge t^{-m},  \qquad t\in (0,t_m). \label{6.49}
\end{equation}
As  in the previous step  we  assume  (\ref{6.19}) as  well. 
By  the  definition of $\varphi(\rho)$ and (\ref{6.49}), we  have  
$$\varphi'(\rho)= w(\varphi(\rho))\ge \varphi(\rho)^{-m}, \qquad \rho \in (0, \varphi^{-1}(t_m)).$$ 
By integrating this over an interval $(0,\rho)$ we  have 
\begin{equation}
\varphi(\rho) \ge (m+1)^{1/(m+1)} \rho^{1/(m+1)}, \qquad \rho \in (0, \varphi^{-1}(t_m)).
\end{equation}
Combining  this with (\ref{6.47}) we have 
$$(m+1)^{1/(m+1)}\le \frac{\varphi(\tilde\eta)}{{\tilde\eta}^{C_0}} \rho^{C_0-1/(m+1)}, \qquad \rho \in (0, \varphi^{-1}(t_m)).$$
But this does not hold if $m$ is  sufficiently large, and hence the assertion is proved.
\end{proof}

\section{Appendix; Some relations among the best constants}
In this section we review fundamental properties of the best constants $S^{p,q;\gamma}$, $S^{p,q;\gamma}_{\rm rad}$, $C^{p,q; R}$  and  $C^{p,q; R}_{\rm rad}$. 
Most of the contents are borrowed from \cite{hk3} (See Section 2.2 and Section 2.3).
Let us introduce  some notations.
\begin{df}\label{df2.4}
For
$1<p\le q<\infty{\;\!}$,  we set
\begin{equation}\dstl{
 \gamma_{p\;\!\!,q}^{}=\dfrac{n-{\;\!\!}1}{1{\;\!\!}+q/p'\;\!\!},{\,\,}
 S_{\;\!\!p\;\!\!,q}^{}=\left\{\begin{array}{cl}\dstl{
  ({\:\!}p')^{p\:\!-\:\!2\:\!+\:\!p\;\!\!/\;\!\!q}q^{\:\!p\;\!\!/\;\!\!q}
  \left(\:\!\!\dfrac{\omega_{n}^{}}{\tau_{\;\!\!p\;\!\!,q}^{}\;\!\!}{\:\!}
  {\rm B}\left(\;\!\!\dfrac{1}{p{\;\!}\tau_{\;\!\!p\;\!\!,q}^{}\;\!\!},
  \dfrac{1}{p'\tau_{\;\!\!p\;\!\!,q}^{}\;\!\!}\:\!\!\right)\:\!\!\!
  \right)^{\!1-\:\!p\;\!\!/\;\!\!q}
 } & \mbox{ if }{\:\!}p<q, \smallskip \\ 1 & \mbox{ if }{\:\!}p=q. \end{array}\right.
}
\end{equation}
Here $\tau_{p,q}=1/p-1/q $ and ${{\rm B}(\:\!\cdot\:\!,\:\!\cdot\:\!)}$ is the beta function.
\end{df}
 \begin{prop} ( Non-critical   CKN-type inequalities )\label{classicalCKN}
 \par Assume that $1<p\le q<\infty$,  $\tau_{\;\!\!p\;\!\!,q}^{}\le1/n$  and $\gamma\neq0$. 
 Then  we have  the followings:
 \begin{enumerate}
 \item 
$S_{\rm{rad}}^{p,q;\gamma}\ge S^{p,q;\gamma}>0 $.
\item
$\dstl{
 S^{\:\!p\;\!\!,q\:\!;\:\!\gamma}=S^{\:\!p\;\!\!,q\:\!;\:\!-\:\!\gamma} \text {and } S_{\rm{rad}}^{\:\!p\;\!\!,q\:\!;\:\!\gamma}=S_{\rm{rad}}^{\:\!p\;\!\!,q\:\!;\:\!-\:\!\gamma}.
}$
\item
$\dstl{
 S^{\:\!p\;\!\!,q\:\!;\:\!\gamma}_{\rm rad}
 =S_{\;\!\!p\;\!\!,q}^{}|\gamma{\:\!}|^{p\:\!(1-\:\!\tau_{\;\!\!p\:\!\!,\;\!\!q}^{}\;\!\!)}.
}$
\item
$S^{p,q;\gamma}= S_{\rm rad}^{p,q;\gamma}= S_{p,q}|\gamma|^{p(1-\tau_{p,q})}$ for
$ 0<|\gamma|\le \gamma_{p,q}$. 

\end{enumerate}
 \end{prop}

\begin{df}\label{df2.6}
For $1<p\le q<\infty{\;\!}$
we set
\begin{equation}
 R_{\:\!p\;\!\!,q}^{}=\exp\dfrac{1{\;\!\!}+q/p'}{\;\!\!(n-{\;\!\!}1){\:\!}p'}{\,\,\,}
 \mbox{ if }{\:\!}n\ge 2, \quad{\,\,\,\,}C_{\;\!\!p\;\!\!,q}^{}
 =S_{\;\!\!p\;\!\!,q}
 {\;\!\!({\:\!}p')^{p\:\!(\tau_{\;\!\!p\;\!\!,q}\;\!\!-1)}}.
\end{equation}
\end{df}

 \begin{prop} ( Critical   CKN-type inequalities )\label{classicalCKN2}
 Assume that
$1<p\le q<\infty$, $\tau_{\;\!\!p\;\!\!,q}^{}\le1/n$  and  $R>1{\;\!}$. Then, we have:
\begin{enumerate}
\item
${\;\!}\dstl{
 C^{\:\!p\;\!\!,q\:\!;R}_{\rm rad}
}\ge\dstl{
 C^{\:\!p\;\!\!,q\:\!;R}
}>0$
\item
$\dstl{
 C^{\:\!p\;\!\!,q\:\!;R}_{\rm rad}
 =C_{\;\!\!p\;\!\!,q}^{} \hspace{3.46cm} \mbox{ for }R\ge 1.
}$
\item
$\dstl{
 C^{\:\!p\;\!\!,q\:\!;R}=C^{\:\!p\;\!\!,q\:\!;R}_{\rm rad}
 =C_{\;\!\!p\;\!\!,q}^{} \hspace{0.58cm} 
\mbox{ for }R\ge
 R_{\:\!p\;\!\!,q}^{}  \mbox{ if }{\:\!}p\ge n\ge 2.
}$

\end{enumerate}
 \end{prop}

\par\medskip

From Proposition \ref{classicalCKN} and Proposition \ref{classicalCKN2}, noting  that $S_{p,q} {(p')^{p(\tau_{\;\!\!p\;\!\!,q}-1)}}= C_{p,q}$ by Definition \ref{df2.6}, we have an interesting relation among the  best constans:

\begin{prop}\label{relation}\begin{enumerate}
\item
Assume that $n\ge 1$, $1<p\le q<\infty$,  $\tau_{\;\!\!p\;\!\!,q}^{}\le1/n$ and $R\ge 1$. Then
\begin{equation}S_{\rm rad}^{p,q; 1/{p'}}=\frac{ S_{p,q}}{|p'|^{p(1-\tau_{p,q})}}= C_{p,q}=  C^{\:\!p\;\!\!,q\:\!;R}_{\rm rad}
\end{equation}
\item
 If  $n\ge 2$, $1<p\le q<\infty$, $\tau_{p,q}\le 1/n$, $1/p'\le \gamma_{p,q}$  and  $ R\ge R_{p,q}$, then we  have the  relation 
\begin{equation}
S^{p,q;  1/p'}=S^{p,q;  1/p'}_{\rm rad}=
C^{\:\!p\;\!\!,q\:\!;R}_{\rm rad}
=
C^{\:\!p\;\!\!,q\:\!;R}.
\end{equation}
\end{enumerate}
\end{prop}

\bigskip\noindent
{\large\bf Toshio Horiuchi\\Department of Mathematics\\Faculty of Science \\ Ibaraki University\\
Mito, Ibaraki, 310, Japan}\par\noindent
e-mail: toshio.horiuchi.math@vc.ibaraki.ac.jp
\par\noindent

\end{document}